\documentclass[12pt]{amsart}
\usepackage{fullpage}
\usepackage{amssymb, amsmath, amsfonts, amsthm, xspace}
\usepackage{dsfont, mathtools, mathrsfs, enumitem}
\usepackage{url, cite, thmtools, thm-restate}
\usepackage[pdftex, colorlinks, unicode]{hyperref}
\usepackage{cleveref}
\usepackage[T1]{fontenc}
\usepackage{tikz-cd}

\newtheorem{thm} {Theorem}[section]

\newtheorem{lem} [thm]{Lemma}
\newtheorem{cor} [thm]{Corollary}
\newtheorem{defi}[thm]{Definition}
 
\newtheorem{qst} [thm]{Question} 
 
\newcommand{\rvline}[1]{\hspace*{#1}\vline\hspace*{#1}}

\newcommand{\et}{\mathrm{\acute{e}t}}

\newcommand{\ab}{\mathrm{ab}}
\newcommand{\tor}{\mathrm{tor}}
\newcommand{\red}{\mathrm{red}} 

\newcommand{\piNr}{{\pi^N_1}}

\newcommand{\ba}{\mathbf{a}}
\newcommand{\bb}{\mathbf{b}}

\newcommand{\dual}{\vee}
\newcommand{\cont}{\mathrm{cont}}

\DeclareMathOperator{\codim}{codim}

\DeclareMathOperator{\lcm}{lcm} 

\DeclareMathOperator{\HP}{HP} 

\DeclareMathOperator{\Proj}{Proj}
\DeclareMathOperator{\Sym}{Sym}
\DeclareMathOperator{\Spec}{Spec}

\DeclareMathOperator{\coker}{coker}

\DeclareMathOperator{\init}{in}
\DeclareMathOperator{\NS}{NS} 
\DeclareMathOperator{\Fitt}{Fitt}
\DeclareMathOperator{\NNS}{\mathbf{NS}}
\DeclareMathOperator{\Gr}{\mathbf{Gr}}
\DeclareMathOperator{\Pic}{\mathbf{Pic}}

\DeclareMathOperator{\Jac}{\mathbf{Jac}}
\DeclareMathOperator{\pic}{\mathrm{Pic}}
\DeclareMathOperator{\Cl}{\mathrm{Cl}}
\DeclareMathOperator{\Hilb}{\mathbf{Hilb}}
\DeclareMathOperator{\CDiv}{\mathbf{CDiv}}
\DeclareMathOperator{\LinSys}{\mathbf{LinSys}}

\DeclareMathOperator{\Hom}{Hom}
\DeclareMathOperator{\charac}{char}
\DeclareMathOperator{\imag}{im}

\DeclareMathOperator{\order}{{\#}}

\begin{document}
\author{Hyuk Jun Kweon}
\address{Department of Mathematics, Massachusetts Institute of Technology, Cambridge, MA
02139-4307, USA}
\email{kweon@mit.edu}
\urladdr{https://kweon7182.github.io/}
\date{\today}
\thanks{This research was supported in part by Samsung Scholarship and a grant from the Simons Foundation (\#402472 to Bjorn Poonen, and \#550033).}
\title{Bounds on the Torsion Subgroup Schemes of N\'eron--Severi Group Schemes}

\begin{abstract}
  Let $X \hookrightarrow \mathbb{P}^r$ be a smooth projective variety defined by homogeneous polynomials of degree $\leq d$ over an algebraically closed field. Let $\Pic X$ be the Picard scheme of $X$. Let $\Pic^0 X$ be the identity component of $\Pic X$. The N\'eron--Severi group scheme of $X$ is defined by $\NNS X = (\Pic X)/(\Pic^0 X)_\red$. We give an explicit upper bound on the order of the finite group scheme $(\NNS X)_{\tor}$ in terms of $d$ and $r$. As a corollary, we give an upper bound on the order of the finite group $\pi^1_\et(X,x_0)^\ab_\tor$.
  We also show that the torsion subgroup $(\NS X)_\tor$ of the N\'eron--Severi group of $X$ is generated by less than or equal to $(\deg X -1)(\deg X - 2)$ elements in various situations.
\end{abstract}
\maketitle

\section{Introduction}
In this paper, we work over an algebraically closed base field $k$. Although $\charac k$ is arbitrary, we are mostly interested in the case where $\charac k > 0$. The N\'eron--Severi group $\NS X$ of a smooth projective variety $X$ is the group of divisors modulo algebraic equivalence. Thus, we have an exact sequence
\[ 0 \rightarrow \pic^0 X \rightarrow \pic X \rightarrow \NS X \rightarrow 0. \]

N\'eron \cite[p. 145, Th\'eor\`eme 2]{Ner} and Severi \cite{Sev} proved that $\NS X$ is a finitely generated abelian group. Hence, its torsion subgroup $(\NS X)_\tor$ is a finite abelian group. Poonen, Testa and van Luijk gave an algorithm for computing $(\NS X)_\tor$ \cite[Theorem~8.32]{PTL}. The author gave an explicit upper bound on the order of $(\NS X)_\tor$ \cite[Theorem~4.12]{Kwe}.

As in \cite[7.2]{Jak}, define the N\'eron--Severi group scheme $\NNS X$ of $X$ by the exact sequence
\[ 0 \rightarrow (\Pic^0 X)_\red \rightarrow \Pic X \rightarrow \NNS X \rightarrow 0. \]
If $\charac k = 0$, then $\Pic^0 X$ is an abelian variety, so $\NNS X$ is a disjoint union of copies of $\Spec k$. However, if $\charac k = p > 0$, then $\Pic^0 X$ might not be reduced, and Igusa gave the first example \cite{Igu}. Thus, the N\'eron--Severi group scheme may have additional infinitesimal $p$-power torsion.

The torsion subgroup scheme $(\NNS X)_\tor$ of $\NNS X$ is a finite commutative group scheme. It is a birational invariant for smooth proper varieties due to \cite[p.~92, Proposition~8]{Nor} and \cite[Proposition~3.4]{Ant}. The first main goal of this paper is to give an explicit upper bound on the order of $(\NNS X)_\tor$. Let $\exp_a b \coloneqq a^b$.
\begin{restatable*}{thm}{NNSbound}\label{thm:NNS bound}
  Let $X \hookrightarrow \mathbb{P}^r$ be a smooth connected projective variety defined by homogeneous polynomials of degree $\leq d$. Then
  \[ \order(\NNS X)_{\tor} \leq \exp_2\exp_2\exp_2\exp_d\exp_2(2r+6\log_2 r). \]
\end{restatable*} 

One motivation for studying $(\NNS X)_\tor$ is its relationship with fundamental groups. Recall that if $k = \mathbb{C}$, then 
\begin{align*}
  \pi^1(X,x_0)^\ab_\tor &\simeq H_1(X,\mathbb{Z})_\tor\\
                        &\simeq \Hom\!\left(H^2(X,\mathbb{Z})_\tor,\mathbb{Q}/\mathbb{Z}\right) \\
                        &\simeq \Hom((\NS X)_\tor,\mathbb{Q}/\mathbb{Z}).
\end{align*}
However, if $\charac k = p > 0$, then $\pi^1(X,x_0)^\ab_\tor$ is not determined by $(\NS X)_\tor$. Therefore, an upper bound on $\#(\NS X)_\tor$ does not give an upper bound on $\#\pi^1_\et(X,x_0)^\ab_\tor$. Nevertheless, $\pi^1(X,x_0)^\ab_\tor$ is isomorphic to the group of $k$-points of the Cartier dual of $(\NNS X)_\tor$ \cite[Proposition 69]{Jak}. Hence, we give an upper bound on $\#\pi^1_\et(X,x_0)^\ab_\tor$ as a corollary of \Cref{thm:NNS bound}. As far as the author knows, this is the first explicit upper bound on $\#\pi^1_\et(X,x_0)^\ab_\tor$.

\begin{restatable*}{thm}{piBound}\label{thm:pi bound}
  Let $X \hookrightarrow \mathbb{P}^r$ be a smooth connected projective variety defined by homogeneous polynomials of degree $\leq d$ with base point $x_0 \in X(k)$. Then
  \[ \# \pi^\et_1(X,x_0)^\ab_\tor \leq \exp_2\exp_2\exp_2\exp_d\exp_2(2r+6\log_2 r). \]
\end{restatable*}
\noindent Let $\piNr(X,x_0)$ be the Nori's fundamental group scheme \cite{Nor} of $(X,x_0)$. Then we also give a similar upper bound on $\# \piNr(X,x_0)^\ab_\tor$.

The second main goal of this paper is to give an upper bound on the number of generators of $(\NS X)_\tor$. If $\ell \neq \charac k$, then $(\NS X)[\ell^\infty]$ is generated by at most $(\deg X - 1)(\deg X - 2)$ elements by \cite[Corollary~6.4]{Kwe}. The main tool of this bound is the Lefschetz hyperplane theorem on \'etale fundamental groups \cite[XII. Corollaire 3.5]{SGA2}. Similarly, we prove that the Lefschetz hyperplane theorem on $\Pic^\tau X$ gives an upper bound on the number of generators of $(\NS X)[p^\infty]$. However, the author does not know whether this is true in general. Nevertheless, Langer {\cite[Theorem~11.3]{Lan}} proved that if $X$ has a smooth lifting over $W_2(k)$, then the Lefschetz hyperplane theorem on $\Pic^\tau X$ is true.
This gives a bound on the number of generators of $(\NS X)_\tor$.


\begin{restatable*}{thm}{NSGenLifting}\label{thm:NS Gen Bound}
  If $X\hookrightarrow\mathbb{P}^r$ is the reduction of a smooth connected projective scheme $\mathcal{X}\hookrightarrow\mathbb{P}_{W_2(k)}^r$ over $W_2(k)$, then $(\NS X)_\tor$ is generated by $(\deg X - 1)(\deg X - 2)$ elements.
\end{restatable*}

In \Cref{sec:numerical conditions}, we prove that $\# (\NNS X)_\tor \leq \dim_k \Gamma(\Hilb_Q X,\mathcal{O}_{\Hilb_Q X})$ for some explicit Hilbert polynomial $Q$. \Cref{sec:length bound} bounds $\dim_k \Gamma(Y,\mathcal{O}_Y)$ for an arbitrary projective scheme $Y\hookrightarrow\mathbb{P}^r$ defined by homogeneous polynomials of degree at most $d$. \Cref{sec:hilbert scheme} gives an upper bound on $\#(\NNS X)_\tor$. \Cref{sec:fundamental group} gives an upper bound on $\# \pi_1^\et(X,x_0)^\ab_\tor$ and $\# \piNr(X,x_0)^\ab_\tor$. \Cref{sec:Lefschetz} discusses the Lefschetz-type theorem on $\Pic^\tau X$. \Cref{sec:generator bound} proves that $(\NS X)_\tor$ is generated by less than or equal to $(\deg X-1)(\deg X-2)$ elements if $X$ has a smooth lifting over $W_2(k)$.

\section{Notation}
Given a scheme $X$ over $k$, let $\mathcal{O}_X$ be the structure sheaf of $X$. We sometimes denote $\Gamma(X) = \Gamma(X,\mathcal{O}_X)$. Let $\pic X$ be the Picard group of $X$. If $X$ is integral and locally noetherian, then $\Cl X$ denotes the Weil divisor class group of $X$. If $Y$ is a closed subscheme of $X$, let $\mathscr{I}_{Y \subset X}$ be the sheaf of ideal of $Y$ on $X$. If $X$ is a projective space and there is no confusion, then we may let $\mathscr{I}_Y = \mathscr{I}_{Y \subset X}$. Given a $k$-scheme $T$ and a $k$-algebra $A$, let  $X(T) = \Hom(T,X)$ and $X(A) = \Hom(\Spec A,X)$.

Suppose that $X$ is a projective scheme in $\mathbb{P}^r$. Let $\Hilb_Q X$ be the Hilbert scheme of $X$ corresponding to a Hilbert polynomial $Q$. Let $\Hilb^P X$ be the Hilbert scheme of $X$ parametrizing closed subschemes $Z$ of $X$ such that $\mathscr{I}_Z$ has Hilbert polynomial $P$. In particular,
\[\Hilb^{P(n)} X = \Hilb_{\binom{n+r}{r}-P(n)} X.\]
Now, suppose that $X$ is smooth. Let $\CDiv_Q X$ be the subscheme of $\Hilb_Q X$ parametrizing divisors with Hilbert polynomial $Q$. 
If $D \subset \mathbb{P}^r$ is an effective divisor on $X$, let $\HP_D$ be the Hilbert polynomial of $D$ as a subscheme.

Let $\Pic X$ be the Picard scheme of $X$. Let $\Pic^0 X$ be the identity component of $\Pic^0 X$. Let $\Pic^\tau X$ be the disjoint union of the connected components of $\Pic X$ corresponding to $(\NS X)_\tor$. Let $\pic^0 X = (\Pic^0 X)(k)$ and $\pic^\tau X = (\Pic^\tau X)(k)$. Let $\NNS X = \Pic X/(\Pic^0 X)_\red$, and let $\NS X = \pic X/\pic^0 X$. Given $x_0 \in X(k)$, let $\pi^\et_1(X,x_0)$ and $\piNr(X,x_0)$ be the \'etale fundamental group and Nori's fundamental group scheme \cite{Nor} of $(X,x_0)$, respectively. Given a vector space $V$, let $\Gr(n,V)$ be the Grassmannian parametrizing $n$-dimensional linear subspaces of $V$.

Let $S$ be a separated $k$-scheme of finite type, and let $X$ be a $S$-scheme, with the structure map $f\colon X \rightarrow S$. Then given a point $t \in S$, let $X_t$ be the fiber over $t$. Let $\mathscr{F}$ be a quasi-coherent sheaf on $X$, then let $\mathscr{F}_t$ be the sheaf on $X_t$ which is the fiber of $\mathscr{F}$ over $t$. Let $g\colon T \rightarrow S$ be a morphism. Let $p_1\colon X\times_S T \rightarrow X$ and $p_2\colon X\times_S T \rightarrow T$ be the projections. Then as in as in \cite[Definition 9.3.12]{Kle2},  for an invertible sheaf $\mathscr{L}$ on $X$, let $\LinSys_{\mathscr{L}/X/S}$ be the scheme given by
  \begin{align*}
    \LinSys_{\mathscr{L}/X/S}(T) = \big\{ D \,\big|\,
    &\text{$D$ is a relative effective divisor on $X_T/T$ such that}\\
    &\text{$\mathcal{O}_{X_T}(D) \simeq p_1^*\mathscr{L}\otimes p_2^*\mathscr{N}$ for some invertible sheaf $\mathscr{N}$ on $T$}\big\}.
  \end{align*}

  Given a set $S$ and $T$, let $S \amalg T$ be the disjoint union of $S$ and $T$. Let $\# S$ be the cardinality of $S$. Let $\mathbb{N}$ be the set of nonnegative integers, and $\mathbb{N}_{< n}$ be the set of nonnegative integers strictly less than $n$. We regard an $n$-tuple $\ba$ on $S$ as a function $\ba\colon \mathbb{N}_{< n} \rightarrow S$. Thus, the $i$th component of $\ba$ is denoted by $\ba(i)$. 

Given a group $G$, let $G^\ab$ be the abelianization of $G$. Suppose that $G$ is abelian. Then let $G_\tor$, $G[n]$ and $G[p^\infty]$ be the torsion subgroup, $n$-torsion subgroup and $p$-power torsion subgroup of $G$, respectively. For a finite abelian group $G$, let $G^\dual$ equals $\Hom(G,\mathbb{Q}/\mathbb{Z})$.

Given a commutative group scheme $G$,
let $G_\tor$, $G[n]$ and $G[p^\infty]$ be the torsion subgroup scheme, $n$-torsion subgroup scheme, and $p$-power torsion subgroup scheme of $G$, respectively. Suppose that $G$ is finite. Then let $\order G$ be the order of $G$, let $G^\ab$ be the abelianization of $G$, and let $G^\dual$ be the Cartier dual of $G$.

Given a finite module $M$, let $\Fitt_i(M)$ be the $i$th fitting ideal of $M$ \cite[Proposition 20.4]{Eis2}. If $M$ is a graded module, its degree $t$ part is denoted as $M_t$. If $M$ is a finite $k[x_0,x_1,\dots,x_r]$-module, then let $\widetilde{M}$ be the coherent sheaf on $\mathbb{P}^r$ associated to $M$. Given a monomial order $\preceq$ and an ideal $I \subset k[x_0,x_1,\dots,x_r]$, let $\init_\preceq(I)$ be the initial ideal of $I$ with respect to $\preceq$. We write an $m\times n$ matrix as
\[(a_{i,j})_{i<m,j<n} \coloneqq  \begin{pmatrix}
  a_{0,0} & a_{0,1} & \dots & a_{0,n-1}\\ 
  a_{1,0} & a_{1,1} & \dots & a_{1,n-1}\\ 
  \vdots & \vdots & \ddots & \vdots\\ 
  a_{m-1,0} & a_{m-1,1} & \dots & a_{m-1,n-1}
\end{pmatrix}.\]
Let $A$ and $B$ be are commutative rings, and $f\colon A \rightarrow B$ be a homomorphism. Let $M = (a_{i,j})_{i<m,j<n} \in A^{m \times n}$. Let $\imag M$ be the $A$-module generated by the columns of $M$, and let
\[f(M) \coloneqq (\varphi(a_{i,j}))_{i<m,j<n}. \]
Given a linear map $\mu\colon A^m \rightarrow A^l$, let $\mu^{\oplus n}(M)$ be the $l\times n$ matrix obtained by applying $\mu$ to every column vector of $M$.

\section{Numerical Conditions}\label{sec:numerical conditions}
Let $\imath\colon X \hookrightarrow \mathbb{P}^r$ be a closed embedding of a smooth connected projective variety $X$. Let $H \subset X$ be a hyperplane section of $X$. The goal of this section is to describe an integer $m$ such that
\begin{equation}\label{eqn:bound NS tor by Gamma Hilb}
  \order (\NNS X)_\tor \leq \dim_k \Gamma\left(\Hilb_{\HP_{mH}} X\right).
\end{equation}
For a positive integer $m$, consider the diagram
\[ \Hilb_{\HP_{mH}} \hookleftarrow \CDiv_{\HP_{mH}} X \rightarrow \Pic^\tau X \twoheadrightarrow (\NNS X)_\tor, \]
in which the middle morphism is yet to be defined. The natural embedding $\CDiv_{\HP_{mH}} X \hookrightarrow \Hilb_{\HP_{mH}} X$ is open and closed \cite[Theorem~1.13]{Kol}. The natural quotient $\Pic^\tau X \twoheadrightarrow (\NNS X)_\tor$ is faithfully flat \cite[VI\textsubscript{A}, Th\'eor\`eme 3.2]{SGA3I}. Suppose that there is a faithfully flat morphism $\CDiv_{\HP_{mH}} X \rightarrow \Pic^\tau X$ for some large $m$. Then $\Gamma((\NNS X)_\tor) \hookrightarrow \Gamma(\CDiv_{\HP_{mH}})$ is injective by \cite[Corollaire~2.2.8]{EGA4}, and $\Gamma(\Hilb_{\HP_{mH}}) \twoheadrightarrow \Gamma(\CDiv_{\HP_{mH}})$ is surjective, so we get (\ref{eqn:bound NS tor by Gamma Hilb}). Such $m$ will be obtained as a byproduct of the construction of Picard schemes. The theorem and the proof below are a modification of the proof of \cite[Theorem~9.4.8]{Kle2}.

\begin{thm}\label{thm:EffDiv quotient}
  Let $m$ be an integer such that
  \begin{equation}\label{eqn:condition for m}
    H^1(X,\mathscr{L}(m)) = 0
  \end{equation}
  for every numerically zero invertible sheaf $\mathscr{L}$. Then the morphism given by
  \begin{align*}
    (\CDiv_{\HP_{mH}} X)(T) &\rightarrow (\Pic^\tau X)(T)\\
    D &\mapsto \mathcal{O}_{X_T}(D-mH_T)
  \end{align*}
  is faithfully flat.
\end{thm}
\begin{proof}
  Let $\Pic^\sigma X$ be the union of the connected components of $\Pic X$ parametrizing invertible sheaves having the same Hilbert polynomial as $\mathcal{O}_X(m)$. Then for any $k$-scheme $T$,
  \begin{align*}
    (\Pic^\sigma X)(T) &\rightarrow (\Pic^\tau X)(T)\\
    \mathscr{L} &\mapsto \mathscr{L}(-m)
  \end{align*}
  gives an isomorphism. Because of the cycle map
  \begin{align*}
    (\CDiv_{\HP_{mH}} X)(T) &\rightarrow (\Pic^\sigma X)(T)\\
    D &\mapsto \mathcal{O}_{X_T}(D),
  \end{align*}
  $\CDiv_{\HP_{mH}} X$ can be regarded as a $(\Pic^\sigma X)$-scheme. The identity $\Pic^\sigma X \rightarrow \Pic^\sigma X$ gives a $(\Pic^\sigma X)$-point of $\Pic^\sigma X$, and it is represented by an invertible sheaf $\mathscr{L}$ on $X_{\Pic^\sigma X}$. Then $\CDiv_{\HP_{mH}} X$ is naturally isomorphic to $\LinSys_{\mathscr{L}/X \times \Pic^\sigma X/\Pic^\sigma X}$ as $(\Pic^\sigma X)$-schemes by definition. For every closed point $t \in \Pic^\sigma X$, $\mathscr{L}_t(-m)$ is numerically zero, so
    \[ H^1(X\times\{t\},\mathscr{L}_t) = 0.\]
    Hence, \cite[9.3.10]{Kle2} implies that $\LinSys_{\mathscr{L}/X \times \Pic^\sigma X/\Pic^\sigma X} \simeq \mathbb{P}(\mathcal{Q})$ for some locally free sheaf $\mathcal{Q}$ on $\Pic^\sigma X$. Since $\mathbb{P}(\mathcal{Q}) \rightarrow \Pic^\sigma X$ is faithfully flat, $\CDiv_{\HP_{mH}} X \rightarrow \Pic^\tau X$ is also faithfully flat.
\end{proof}

We now need to find an integer $m$ satisfying (\ref{eqn:condition for m}). Recall the definition of Castelnuovo–Mumford regularity.
\begin{defi}
  A coherent sheaf $\mathscr{F}$ on $\mathbb{P}^r$ is $m$-regular if and only if
  \[ H^i\left(\mathbb{P}^r, \mathscr{F}(m-i)\right) = 0 \]
  for every integer $i > 0$. The smallest such $m$ is called the Castelnuovo–Mumford regularity of $\mathscr{F}$.
\end{defi}

Mumford \cite[p. 99]{Mum} showed that if $\mathscr{F}$ is $m$-regular, then it is also $t$-regular for every $t \geq m$.

\begin{defi}
  Let $P$ be the Hilbert polynomial of some homogeneous ideal $I \subset k[x_0,\dots,x_r]$. The Gotzmann number $\varphi(P)$ of $P$ is defined as 
  \begin{align*}
    \varphi(P) = \inf \{ m \,|\,& \mathscr{I}_Z  \text{ is $m$-regular for every} \\
        &\text{closed subvariety $Z \subset \mathbb{P}^r$ with Hilbert polynomial $P$} \}.
  \end{align*}
  If $Z \subset \mathbb{P}^r$ is a projective scheme, then let $\varphi(Z)$ be the Gotzmann number of the Hilbert polynomial of $\mathscr{I}_Z$. 
\end{defi}

Given a homogeneous ideal $I \subset k[x_0,\dots,x_r]$, Hoa gave an explicit finite upper bound on $\varphi(\HP_I)$ {\cite[Theorem~6.4(i)]{Hoa}}. We will describe $m$ in terms of Gotzmann numbers of Hilbert polynomials.

\begin{lem}\label{lem:condition for effectiveness}
  Let $\mathscr{L}$ be a numerically zero invertible sheaf on $X$. Then $\mathscr{L}((d-1)\codim X)$ is generated by global sections.
\end{lem}
\begin{proof}
  See \cite[Lemma~3.5 (a)]{Kwe}.
\end{proof}

\begin{lem}\label{lem:eqn 2}
  Let $\mathscr{L}$ be a numerically zero invertible sheaf on $X$. Then for every $i \geq 1$, $n \geq (d-1)\codim X$ and $m \geq \max\{ \varphi(nH), \varphi(X)\}$, we have
  \[ H^i(X,\mathscr{L}(m-n)) = 0. \]
\end{lem}
\begin{proof}
  There is an effective divisor $Z$ on $X$ such that $\mathscr{I}_{Z\subset X} = \mathscr{L}(-n)$ by \Cref{lem:condition for effectiveness}. Since $\mathscr{I}_Z$ and $\mathscr{I}_X$ are $m$-regular, $H^i(\mathbb{P}^r,\mathscr{I}_Z(m)) = 0$ and $H^i(\mathbb{P}^r,\mathscr{I}_X(m)) = 0$ for all $i\geq 1$. Consider the exact sequence
  \[ 0 \rightarrow \mathscr{I}_X \rightarrow \mathscr{I}_Z \rightarrow \imath_* \mathscr{I}_{Z\subset X} \rightarrow 0. \]
  Then its long exact sequence shows that
  \[ H^i(X,\mathscr{L}(m-n)) = H^i(\mathbb{P}^r,\imath_* \mathscr{I}_{Z\subset X}(m)) = 0\]
  for every $i \geq 1$.
\end{proof}

We are ready to prove the main theorem of this section.
\begin{thm}\label{thm:bound NS by Hilb}
  Assume that $n \geq (d-1)\codim X$ and $m \geq \max\{ \varphi({nH}), \varphi(X)\}$. Then
  \[\order (\NNS X)_\tor \leq 
    \dim_k \Gamma\left(\Hilb_{\HP_{mH}} X\right). \]
\end{thm}
\begin{proof}
  The quotient $\Pic^\tau X \rightarrow (\NNS X)_\tor$ is faithfully flat. \Cref{lem:eqn 2} and \Cref{thm:EffDiv quotient} give a faithfully flat morphism $\CDiv_{\HP_{mH}} X \rightarrow \Pic^\tau X$. As a result, their composition gives an injection 
  \[ \Gamma((\NNS X)_\tor) \hookrightarrow \Gamma\left(\CDiv_{\HP_{mH}} X\right) \]
  by \cite[Corollaire 2.2.8]{EGA4}. Recall that $\CDiv_{\HP_{mH}} X$ is an open and closed subscheme of $\Hilb_{\HP_{mH}} X$. Consequently,
  \begin{align*}
    \order (\NNS X)_\tor &= \dim_k \Gamma(\NNS X)_\tor \\
                         &\leq \dim_k \Gamma\left(\CDiv_{\HP_{mH}} X\right) \\
                         &\leq \dim_k \Gamma\left(\Hilb_{\HP_{mH}} X\right). \qedhere
  \end{align*}
\end{proof}

Therefore, if we bound $\dim_k \Gamma(Y,\mathcal{O}_Y)$ for a general projective scheme $Y$, and explicitly describe Hilbert schemes in projective spaces, then we can bound $\order (\NNS X)_\tor$.

\section{The Dimension of \texorpdfstring{$\Gamma(Y,\mathcal{O}_Y)$}{Γ(Y,O\_Y)}}\label{sec:length bound}
Let $Y\hookrightarrow\mathbb{P}^r$ be a projective scheme defined by a homogeneous ideal $I$ generated by homogeneous polynomials of degree at most $d$. The aim of the section is to give an upper bound on $\dim_k \Gamma(Y,\mathcal{O}_Y)$ in terms of $r$ and $d$. First, we reduce to the case where $I$ is a monomial ideal. 

\begin{lem}\label{lem:bounded by init}
  Let $\preceq$ be a monomial order. Let $Y_0\hookrightarrow\mathbb{P}^r$ be a projective scheme defined by $\init_\preceq(I)$. Then
  \[\dim_k \Gamma(Y,\mathcal{O}_Y) \leq \dim_k\Gamma(Y_0,\mathcal{O}_{Y_0}).\]
\end{lem}
\begin{proof}
  By \cite[Corollary~3.2.6]{HeHi} and \cite[Theorem~3.1.2]{HeHi}, there is a flat family $\mathcal{Y} \rightarrow \Spec k[t]$ such that
  \begin{enumerate}
  \item $\mathcal{Y}_0 \simeq Y_0$ and
  \item $\mathcal{Y}_t \simeq Y$ for every $t \neq 0$.
  \end{enumerate}
  Hence, $\dim_k \Gamma(Y,\mathcal{O}_Y) \leq \dim_k\Gamma(Y_0,\mathcal{O}_{Y_0})$ by the semicontinuity theorem \cite[Theorem~12.8]{Har}.
\end{proof}


\begin{thm}[Dub\'e {\cite[Theorem~8.2]{Dub}}]\label{thm:init degree bound}
  Let $I \subset k[x_0,\dots,x_r]$ be an ideal generated by homogeneous polynomials of degree at most $d$. Let $\preceq$ be a monomial order. Then $\init_{\preceq}(I)$ is generated by monomials of degree at most
  \[ 2\left(\frac{d^2}{2}+d\right)^{2^{r-1}}. \]
\end{thm}

Hence, we will focus on the case where $I$ is a monomial ideal. Let $m$ be a monomial and $I$ be a monomial ideal of $ k[x_0,\dots,x_r]$. Then by abusing notation, we denote
\[ \frac{(m)}{I} \coloneqq \frac{(m)}{I\cap(m)}. \]

\begin{defi}
  Let $m \in k[x_0,\dots,x_r]$ be a monomial, and let $I \subset k[x_0,\dots,x_r]$ be a monomial ideal. Let $M(m,I)$ be the set of monomials in $(m)\setminus (I \cup (x_0^d,\dots,x_r^d))$. In particular,
\[ \# M(m,I) = \dim_k \frac{(m)}{I+(x_0^d,\dots,x_r^d)}. \] 
\end{defi}


\begin{defi}
  Given a monomial $m \in k[x_0,\dots,x_r]$, let $\mu(m)+1$ be the number of $i$ such that $x_i\mid m$. 
\end{defi}
\begin{defi}
  Let $I$ be a monomial ideal with minimal monomial generators $m_0,m_1,\dots, m_{n-1}$. Then
  \[\mu(I) \coloneqq \mu(m_0)+\mu(m_1)+\dots+\mu(m_{n-1}). \]
\end{defi}

Suppose that $I$ is generated by monomials of degree at most $d$. We want to show that $\dim_k \Gamma(Y,\mathcal{O}_Y) \leq d^r$. By induction on $\mu(I)$, we will show a slightly stronger proposition: for every monomial $m \in k[x_0,\dots,x_r]$ whose exponents are at most $d$, we have
\[\dim_k \Gamma\left(\mathbb{P}^r,\widetilde{(m)/I}\right) \leq \frac{\# M(m,I)}{d}.\]
Since $\Gamma(Y,\mathcal{O}_Y) = \Gamma(\mathbb{P}^r,\widetilde{(1)/I})$ and $\#M(m,I) \leq d^{r+1}$, it follows that $\dim_k \Gamma(Y,\mathcal{O}_Y) \leq d^r$.

\begin{lem}\label{lem:trivial case}
  For some $n \leq r+1$, let $I = (x_0^{b_0},x_1^{b_1},\dots,x_{n-1}^{b_{n-1}})$ such that $b_i > 0$ for every $i < n$. Then
  \[ \dim_k \Gamma(Y,\mathcal{O}_Y) = 
    \begin{cases}
      1                & \text{if $n < r$,} \\
      b_0\dots b_{n-1} & \text{if $n = r$, and} \\
      0                & \text{if $n = r+1$.}
    \end{cases}      \]
\end{lem}
\begin{proof}
  Suppose that $n < r$, and take a nonzero function $f \in \Gamma(Y,\mathcal{O}_Y)$. Take any $i$ and $j$ such that $n \leq i < j \leq r$. In the open set $x_i \neq 0$, we have $f = g/x_i^{\deg g}$ such that $g \in k[x_0,\dots,x_r]$ and $x_i\nmid g$. In the open set $x_j \neq 0$, we have $f = h/x_j^{\deg h}$ such that $h \in k[x_0,\dots,x_r]$ and $x_j\nmid h$. Therefore, in the open set $x_ix_j \neq 0$, we have
  \[ g x_j^{\deg h} = h x_i^{\deg g}. \]
  Therefore, $\deg g = \deg h = 0$. As a result, $f$ is constant meaning that $\dim_k \Gamma(Y,\mathcal{O}_Y)=1$.

  If $n = r$, then $\dim_k \Gamma(Y,\mathcal{O}_Y) = b_0\dots b_{n-1}$ by B\'ezout's theorem. If $n = r+1$, then $Y$ is empty, so $\dim_k \Gamma(Y,\mathcal{O}_Y) = 0$.
\end{proof}

\begin{lem}\label{lem:induction step}
  Let $m$ be a monomial and $I$ be a monomial ideal such that $\mu(I) > 0$. Then there are monomials $m_0$ and $m_1$, and monomial ideals $I_0$ and $I_1$ such that
  \begin{align*}
    M(m,I) &= M(m_0,I_0) \amalg M(m_1,I_1),  \\
    \dim_k \Gamma\left(\mathbb{P}^r,\widetilde{(m)/I}\right)
           &\leq \dim_k \Gamma\left(\mathbb{P}^r,\widetilde{(m_0)/I_0}\right)
             +\dim_k \Gamma\left(\mathbb{P}^r,\widetilde{(m_1)/I_1}\right),\\
           \mu(I_0) &< \mu(I)\text{ and } \mu(I_1) < \mu(I).
  \end{align*}
\end{lem}
\begin{proof}
  Since $\mu(I) > 0$, there is a minimal monomial generator $n'$ of $I$ such that $\mu(n') > 0$. Let $J$ be the ideal generated by the other minimal monomial generators of $I$. Then $I = (n') + J$ and $n' \not\in J$. We may assume that $n' = x_i^s n$ where $s > 0$ and $x_i\nmid n$. Grouping the monomials in $M(m,I)$ according to whether the exponent of $x_i$ is $\geq s$ or $<s$ yields
  \begin{align*}
    M(m,I) &= M\left(\lcm(x_i^s,m),(n)+J\right) \amalg M\left(m,(x_i^s)+J\right).
  \end{align*}
  Moreover, the short exact sequence
  \[ 0 \rightarrow \frac{\left(\lcm(x_i^s,m)\right)}{(n)+J}
    \rightarrow \frac{(m)}{I}
    \rightarrow \frac{(m)}{(x_i^s)+J} \rightarrow 0. \]
  implies that
  \[\dim_k \Gamma\left(\mathbb{P}^r,\widetilde{(m)/I}\right) \leq
    \dim_k \Gamma\left(\mathbb{P}^r,\widetilde{\frac{(\lcm(x_i^s,m))}{(n)+J}}\right) +
    \dim_k \Gamma\left(\mathbb{P}^r,\widetilde{\frac{(m)}{(x_i^s)+J}}\right). \]
  Finally, we have
  \[ \mu((n)+J) \leq \mu(n) + \mu(J) < \mu(n') + \mu(J) = \mu(I) \]
  and similarly $\mu((x_i^s)+J) < \mu(I)$.
\end{proof}

\begin{lem}\label{lem:stronger monomial bound}
  Let $m\in k[x_0,\dots,x_r]$ be a monomial whose exponents are at most $d$. Let $I\subset k[x_0,\dots,x_r]$ be an ideal generated by monomials of degree at most $d$. Then 
  \[ \dim_k \Gamma\left(\mathbb{P}^r,\widetilde{(m)/I}\right) \leq \frac{\# M(m,I)}{d}. \]
\end{lem}
\begin{proof}
  This will be shown by induction on $\mu(I)$. If $\mu(I)<0$, then $I = (1)$, so
  \[\dim_k \Gamma\left(\mathbb{P}^r,\widetilde{(m)/I}\right) = 0.\]
  
  Suppose that $\mu(I)=0$. Then by changing coordinates, we may assume that for some $n \leq r+1$,
  \[I = (x_0^{b_0},x_1^{b_1},\dots,x_{n-1}^{b_{n-1}}) \text{ and } m = x_0^{a_0}x_1^{a_1}\dots x_{r}^{a_r}\]
  such that $0 < b_i \leq d$ for every $i < n$, and $0 \leq a_j \leq d$ for every $j \leq r$. If $a_i \geq b_i$ for some $i<n$, then $m \in I$, so $(m)/I = 0$. Thus, we may assume that $a_i < b_i \leq d$ for every $i < n$. Let $\imath\colon Y \hookrightarrow \mathbb{P}^r$ be the projective scheme defined by $I$. Then $\widetilde{(m)/I}$ is a subsheaf of $\imath_* \mathcal{O}_Y$. Hence, if $n < r$, then \Cref{lem:trivial case} implies that
  \[ \dim_k \Gamma\left(\mathbb{P}^r,\widetilde{(m)/I}\right) \leq 1.
  \]
  Suppose that $n = r$. Since the open set defined by $x_r \neq 0$ contains $Y$, multiplying by $x_r^{\deg m}$ gives an isomorphism
  \[ \times x_r^{\deg m} \colon \widetilde{\frac{(m)}{I}} \rightarrow \widetilde{\frac{(m)}{I}}(\deg m).\]
  Moreover, multiplication by $m^{-1}$ gives an isomorphism 
  \[ \times m^{-1}\colon {\frac{(m)}{I}}(\deg m) \rightarrow {\frac{k[x_0,\dots,x_r]}{(x_0^{b_0-a_0},\dots,x_{r-1}^{b_{r-1}-a_{r-1}})}}.\]
  Hence, by \Cref{lem:trivial case},
  \[ \dim_k \Gamma\left(\mathbb{P}^r,\widetilde{(m)/I}\right) = (b_0-a_0) (b_1-a_1) \dots (b_{r-1}-a_{r-1}).\]
  Suppose that $n = r+1$. Then since $Y$ is empty,
  \[\dim_k \Gamma\left(\mathbb{P}^r,\widetilde{(m)/I}\right) = 0.\]
  Because
  \[\# M(m,I) = (b_0-a_0) (b_1-a_1) \dots (b_{n-1}-a_{n-1}) d^{r-n+1}, \]
  we have $\Gamma\left(\mathbb{P}^r,\widetilde{(m)/I}\right) \leq \# M(m,I)/{d}$ in every case.

  Now, suppose that $\mu(I) > 0$ and assume the induction hypothesis. Then we have $m_0$, $m_1$, $I_0$ and $I_1$ as in \Cref{lem:induction step}. Thus,
  \begin{align*}
    \Gamma\left(\mathbb{P}^r,\widetilde{(m)/I}\right)
    &\leq \Gamma\left(\mathbb{P}^r,\widetilde{(m_0)/I_0}\right) + \Gamma\left(\mathbb{P}^r,\widetilde{(m_0)/I_0}\right)\\
    &\leq \frac{\# M(m_0,I_0)}{d} +\frac{ \# M(m_1,I_1)}{d} \text{ (by the induction hypothesis)}\\
    &= \frac{\# M(m,I)}{d}. \qedhere
  \end{align*}
\end{proof}

\begin{cor}\label{lem:monomial ideal bound}
  Let $Y\hookrightarrow\mathbb{P}^r$ be a projective scheme defined by monomials of degree at most $d$. Then
  \[ \dim_k \Gamma(Y,\mathcal{O}_Y) \leq d^r. \]
\end{cor}
\begin{proof}
  Let $I$ be the ideal generated by the defining monomials of $Y$. Then 
  \begin{align*}
    \dim_k \Gamma(Y,\mathcal{O}_Y)
    &= \dim_k \Gamma\left(\mathbb{P}^r,\widetilde{(1)/I}\right)\\
    &\leq \frac{\# M(1,I)}{d} \text{ (by \Cref{lem:stronger monomial bound} with $m = 1$)}\\
    &= d^r. \qedhere
  \end{align*}
\end{proof}

Now, we are ready to prove the main theorem of this section.

\begin{thm}\label{thm:bound on Gamma}
  Let $Y\hookrightarrow\mathbb{P}^r$ be a projective scheme defined by homogeneous polynomials of degree at most $d$. Then
  \[ \dim_k \Gamma(Y,\mathcal{O}_Y) \leq 2^r \left(\frac{d^2}{2}+d\right)^{r 2^{r-1}}\]
\end{thm}
\begin{proof}
  This follows from \Cref{lem:bounded by init}, \Cref{thm:init degree bound} and \Cref{lem:monomial ideal bound}.
\end{proof}

The artinian scheme $\Spec \Gamma(Y,\mathcal{O}_Y)$ satisfies the universal property: every morphism $f\colon Y\rightarrow G$ to an artinian scheme $G$ uniquely factors through the natural morphism $p\colon Y \rightarrow \Spec \Gamma(Y,\mathcal{O}_Y)$ \cite[Exercise II.2.4]{Har}. Thus, $Y$ and $\Spec \Gamma(Y,\mathcal{O}_Y)$ has the same number of connected components. Andreotti--Bézout inequality \cite[Lemma~1.28]{Cat} then implies that $\dim_k \Gamma(Y,\mathcal{O}_Y)_\red \leq d^r$. Thus, we may expect that $\dim_k \Gamma(Y,\mathcal{O}_Y) \leq d^r$, and this is true if $Y$ is defined by monomials by \Cref{lem:monomial ideal bound}.
\begin{qst}
  Let $Y\hookrightarrow\mathbb{P}^r$ be a projective scheme defined by homogeneous polynomials of degree at most $d$. Is
  \[\dim_k \Gamma(Y,\mathcal{O}_Y) \leq d^r?\]
\end{qst}
One the other hand, Mayr and Meyer \cite{MaMe} constructed a family of ideals whose Castelnuovo--Mumford regularities are doubly exponential in $r$. Thus, the doubly exponential upper bound on $\dim_k \Gamma(Y,\mathcal{O}_Y)$ might be unavoidable.

\section{Hilbert Schemes}\label{sec:hilbert scheme}



Let $X \hookrightarrow \mathbb{P}^r$ be a smooth connected projective variety defined by polynomials of degree at most $d$. The aim of this section is to give an explicit upper bound on $\#(\NNS X)_{\tor}$. \Cref{thm:bound NS by Hilb} implies that it suffices to give an upper bound on $\dim_k \Gamma(\Hilb^P X)$ for a polynomial $P$. Gotzmann explicitly described $\Hilb^P \mathbb{P}^r$ as a closed subscheme of a Grassmannian \cite{Got}\cite[Proposition~C.29]{IaKa}. Let $R_n \coloneqq k[x_0,\dots,x_r]_n$ for $n \in \mathbb{N}$.

\begin{thm}[{Gotzmann}]\label{thm:explicit construction of Hilb P}
  Let $t \geq \varphi(P)$ be an integer. Then there exists a closed immersion given by
  \begin{align*}
    \left(\Hilb^P \mathbb{P}^r\right)(A) &\rightarrow \Gr(P(t),R_t)(A) \\
    [Z] &\mapsto \Gamma(\mathbb{P}^r_A,\mathscr{I}_Z(t))
  \end{align*}
  for every $k$-algebra $A$.
\end{thm}

  Therefore, we have closed immersions
\[ \Hilb^{P} X \hookrightarrow \Hilb^{P} \mathbb{P}^r
  \hookrightarrow \Gr(P(t),R_t)
  \hookrightarrow \mathbb{P}\!\left({\textstyle \bigwedge^{P(t)}} R_t\right)\]
where $\Hilb^{P} X \hookrightarrow \Hilb^{P} \mathbb{P}^r$ is the natural embedding and $\Gr(P(t),R_t)
\hookrightarrow \mathbb{P}(\bigwedge^{P(t)}R_t)$ is the Pl\"ucker embedding. We will bound the degree of defining equations of $\Hilb^{P} X$ in $\mathbb{P}(\bigwedge^{P(t)}R_t)$. Then \Cref{thm:bound on Gamma} will gives an upper bound on $\dim_k \Gamma(\Hilb^P X)$. For simplicity, let
\[Q(n) \coloneqq \dim R_n - P(n) = \binom{n+r}{r}-P(n).\]
The theorem below is a refomulation of the work in \cite[Appendix~C]{IaKa}. 


\begin{thm}\label{thm:explicit construction of Hilb X}
  Let $t \geq \max\{\varphi(P),d\}$ be an integer. Let $A$ be an $k$-algebra, and let $S_n = A[x_0,\dots,x_r]_n$. Then $M \in \Gr(P(t),R_t)(A)$ is in $(\Hilb^P X)(A)$ if and only if
  \begin{enumerate}[label=(\alph*)]
  \item $\Fitt_{Q(t+1)-1} (S_t / (S_1\cdot M)) = 0$ \text{ and}\label{item:Hilb P}
  \item $\Gamma(\mathbb{P}^r_A,\mathscr{I}_{X_A}(t)) \subset M$.\label{item:Hilb X}
  \end{enumerate}
\end{thm}
\begin{proof}
  Note that $M \subset S_t$ is a projective $A$-module of rank $P(t)$. Nakayama's lemma and \cite[Proposition~C.4]{IaKa} imply that $S_t / (S_1\cdot M)$ is locally generated by $Q(t+1)$ elements. Thus, \cite[Proposition 20.6]{Eis2} implies that $\Fitt_{Q(t+1)} (S_t / (S_1\cdot M)) = A$. Hence, by \cite[Proposition 20.8]{Eis2}, $M$ satisfies \ref{item:Hilb P} if and only if $S_1\cdot M \in \Gr(P(t+1),R_{t+1})(A)$, which by \cite[Proposition~C.29]{IaKa} is equivalent to $M\in(\Hilb^P \mathbb{P}^r)(A)$.

  Let $Z$ be the $A$-scheme defined by the polynomials in $M$. Then $M$ satisfies \ref{item:Hilb X} if and only if $Z \subset X_A$. Therefore, $M\in \Gr(P(t),R_t)(A)$ satisfies both \ref{item:Hilb P} and \ref{item:Hilb X} if and only if $M \in (\Hilb^P X)(A)$.
\end{proof}

For simplicity, we replace $R_t$ by a $k$-vector space $V$, and $P(t)$ by a nonnegative integer $n \leq \dim V$. Let $e_0,e_1,\dots,e_{\dim V-1}$ be an ordered basis of $V$. Given $\ba\colon \mathbb{N}_{<n} \rightarrow \mathbb{N}_{<\dim V}$, let 
\[z_\ba \coloneqq e_{\ba(0)} \wedge \dots \wedge e_{\ba(n-1)} \in {\textstyle \bigwedge^n} V. \]
Then
\[\mathbb{P} \!\left({\textstyle \bigwedge^n} V\right) = \Proj \Sym\!\left({\textstyle \bigwedge^n} V\right)
  = \Proj k[\{z_\ba \,|\,\ba\colon \mathbb{N}_{<n} \rightarrow \mathbb{N}_{<\dim V}\}].\]

\begin{defi}
  Given $\ba\colon \mathbb{N}_{<n} \rightarrow \mathbb{N}_{<\dim V}$, let
  \[\ba[j\mapsto i](n) \coloneqq
    \begin{cases}
      \ba(n) & \text{if $n \neq j$ and} \\
      i      & \text{if $n = j$.}
    \end{cases}\]
\end{defi}

\begin{defi}
  Given $\ba\colon \mathbb{N}_{<n} \rightarrow \mathbb{N}_{<\dim V}$, let
  \[K_\ba \coloneqq \left(z_{\ba[j\mapsto i]}\right)_{i<\dim V, j <n}.\]
  Let $\bb_0,\bb_1,\dots,\bb_{(\dim V)^n-1}$ be all the functions $\mathbb{N}_{<\dim V} \rightarrow \mathbb{N}_{<n}$, and let
  \[ L \coloneqq \left(\begin{matrix}
        K_{\bb_0} & \rvline{-\arraycolsep} & K_{\bb_1}
        & \rvline{-\arraycolsep} & \cdots  & \rvline{-\arraycolsep} & K_{\bb_{(\dim V)^n-1}}
      \end{matrix}\right). \]
\end{defi}

Because of the P\"lucker embedding, we can regard $\Gr(n,V)$ as a closed subscheme of $\mathbb{P}(\bigwedge^n V)$. By abusing notation, we regard $z_\ba$ as a global section of $\mathcal{O}_{\Gr(n,V)}(1)$.

\begin{lem}\label{lem:universal bundle}
  Let $\mathcal{B} \hookrightarrow \mathcal{O}_{\Gr(n,V)}^{\dim V}$ be the universal vector bundle on $\Gr(n,V)$. Then $\imag L$ generates $\mathcal{B}(1)$.
\end{lem}
\begin{proof}
   Let $\ba\colon \mathbb{N}_{<n} \rightarrow \mathbb{N}_{<\dim V}$ be an injection. Let $U_\ba\subset\Gr(n,V)$ be the affine open subscheme defined by $z_\ba \neq 0$. Then it suffices to prove that the natural embedding $\imath\colon U_\ba\hookrightarrow\Gr(n,V)$ represents $\imag z_\ba^{-1} L$.

  Without loss of generality, let $\ba(i) = i$ for every $i < n$. By \cite[p. 65]{Harris}, the affine coordinate ring of $U_\ba$ is
  \[ \Gamma(U_\ba) = k[\{x_{i,j} \,|\, n\leq i<\dim V \text{ and } j<n \}], \]
  where $x_{i,j}$ are indeterminate, and $\imath$ represents $\imag J$ such that
  \[J  = 
    \begin{pmatrix}
      1 & \cdots & 0 \\
      \vdots & \ddots & \vdots \\
      0 & \cdots & 1 \\
      x_{n,0}  & \cdots & x_{n,n-1} \\
      \vdots  & \ddots & \vdots \\
      x_{\dim V-1,0} & \cdots & x_{\dim V-1,n-1} \\
    \end{pmatrix}_{\textstyle .}\]
  For simplicity, let $(x_{i,j})_{i<n,j<n}$ be the identity matrix, so $J = (x_{i,j})_{i<\dim V,j<n}$. Given $\bb\colon \mathbb{N}_{<n} \rightarrow \mathbb{N}_{<\dim V}$, let $J_\bb$ be the $n\times n$ matrix such that the $i$th row of $J_\bb$ is the $\bb(i)$th row of $J$. Then $J_\ba$ is the identity matrix. By the definition of the P\"lucker embedding,
 \[\frac{z_\bb}{z_\ba} = \frac{\det J_\bb}{\det J_\ba} = \det J_\bb \in \Gamma(U_\ba).\]
 In particular, $z_{\ba[j\mapsto i]} / z_\ba = x_{i,j}$, meaning that $z_\ba^{-1} K_\ba = J$. Thus, $\imag J \subset \imag z_\ba^{-1} L$, so it suffices to prove that $\imag z_\ba^{-1} K_\bb \subset \imag J$ for every $\bb\colon \mathbb{N}_{<n} \rightarrow \mathbb{N}_{<\dim V}$.

  Given $m < n$, let $v_m$ and $w_m$ be the $m$th column vectors of $J = z_\ba^{-1} K_\ba$ and $z_\ba^{-1} K_\bb$, respectively. Let $C_{i,j}$ be the $(i,j)$ cofactor of the matrix $J_\bb$. Then
\[
  w_m = 
        \frac{1}{z_\ba} \begin{pmatrix}
          z_{\bb[m\mapsto 0]} \\
          \vdots \\
          z_{\bb[m\mapsto \dim_k V - 1]}
        \end{pmatrix}
      = 
        \begin{pmatrix}
          \det J_{\bb[m\mapsto 0]} \\
          \vdots \\
          \det J_{\bb[m\mapsto \dim V - 1]}  \\
        \end{pmatrix} 
      = \sum_{j = 0}^{n-1} C_{m,j} v_j \in \imag J \qedhere
\]
\end{proof}

\begin{lem}\label{thm:bound of H}
  Let $m$ and $q$ be nonnegative intergers. Let $W$ be a $k$-vector space and let $u\colon V^q \rightarrow W$ be a linear map. For every $k$-algebra $A$, let
  \[\mathbf{H}(A) = \{M \in \Gr(n,V)(A) \,|\, \Fitt_{\dim W - m} (W_A / u_A(M^q)) = 0\}, \]
  so $\mathbf{H}$ is a subfunctor of $\Gr(n,V)$.
  Then there is a closed scheme $F \subset \Gr(n,V)$ defined by homogeneous polynomials of degree $m$ such that $\mathbf{H}$ is represented by $G \cap \Gr(n,V)$.
\end{lem}
\begin{proof}
  For $i < q$, let $u_i$ be the $i$th component of $u$, so $u = u_0 \oplus\dots\oplus u_{q-1}$. Let
  \begin{align*}
    \Lambda &\coloneqq \left(\begin{matrix}
        u_0^{\oplus n(\dim V)^n}(L) & \rvline{-\arraycolsep} & u_1^{\oplus n(\dim V)^n}(L)
        & \rvline{-\arraycolsep} & \cdots  & \rvline{-\arraycolsep} & u_{q-1}^{\oplus n(\dim V)^n}(L)
      \end{matrix}\right).
  \end{align*}
  Let $V\hookrightarrow \mathbb{P}\left(\bigwedge^n V\right)$ be the closed subscheme defined by the $m\times m$ minors of $\Lambda$. We will show that $\mathbf{H}$ is represented by $F \cap \Gr(n,V)$.

  Take $U_\ba$ as in the proof of \Cref{lem:universal bundle}. For a $k$-algebra $A$, take an $A$-module $M \in U_\ba(A)$. Then $M$ is represented by a $k$-algebra morphism $f\colon \Gamma(U_\ba) \rightarrow A$. Then \Cref{lem:universal bundle} implies that $M = \imag f(z_\ba^{-1}L)$. Thus,
  \begin{align*}
    u_A(M^q) &= u_{0,A}\!\left(\imag f(z_\ba^{-1}L)\right) + \dots + u_{q-1,A}\!\left(\imag f(z_\ba^{-1}L)\right) \\
             &= \imag f\left(u_0^{\oplus n(\dim V)^n}(z_\ba^{-1}L)\right) + \dots
               + \imag f\left(u_{q-1}^{\oplus n(\dim V)^n}(z_\ba^{-1}L)\right)\\
             &= \imag f(z_\ba^{-1} \Lambda).
  \end{align*}
  Hence, we have a free resolution
  \[\begin{tikzcd}
      \dots\arrow[r] & A^{qn(\dim V)^n}\arrow[r,"f(z_\ba^{-1} \Lambda)"]
      & [1em] W_A\arrow[r] & {\displaystyle \frac{W_A}{u(M^q)}} \arrow[r]& 0.
    \end{tikzcd}\]
  Consequently, $\Fitt_{\dim W - m} (W_A / u(M^s))$ is generated by $m\times m$ minors of $f(z_\ba^{-1} \Lambda)$. As a result, $M \in \mathbf{H}(A)$ if and only if $M \in F(A)$. Since $M$ is arbitrary, this imples that $\mathbf{H}$ is represented by $F \cap \Gr(n,V)$.
\end{proof}

\begin{lem}\label{thm:bound of H2}
  Let $U$ be a linear subspace of $V$. For every $k$-algebra $A$, let
  \[\mathbf{H}(A) = \{M \in \Gr(n,V)(A) \,|\, U_A \subset M\}, \]
  so $\mathbf{H}$ is a subfunctor of $\Gr(n,V)$.
  Then $\mathbf{H}$ is represented by the intersection of a linear space and $\Gr(n,V)$ in $\mathbb{P}\left(\bigwedge^n V\right)$.
\end{lem} 
\begin{proof}
  See \cite[p. 66]{Harris}.
\end{proof}

Hence, we can bound the degree of the defining equation of $\Hilb^{P} X\hookrightarrow\mathbb{P}(\bigwedge^{P(t)}R_t)$.
\begin{thm}\label{lem:Hilb degree}
  The closed embedding $\Hilb^{P} X\hookrightarrow\mathbb{P}(\bigwedge^{P(t)}R_t)$ is defined by homogeneous polynomials of degree at most $P(t+1)+1$. 
\end{thm}
\begin{proof}
  
The Pl\"ucker relations are quadratic \cite[p. 65]{Harris}. Thus, it follows from \Cref{thm:explicit construction of Hilb X}, \Cref{thm:bound of H} with $u(f_0,f_1,\dots,f_r) = x_0f_0+x_1f_1 +\dots+x_r f_r$ and \Cref{thm:bound of H2}.
\end{proof}

Therefore, an upper bound on Gotzmann numbers will give an explicit construction of the Hilbert scheme. Such a bound is given by Hoa {\cite[Theorem~6.4(i)]{Hoa}}.
 
\begin{thm}[Hoa]\label{thm:Gotzmann bound}
  Let $I\subset k[x_0,\dots,x_r]$ be a nonzero ideal generated by homogeneous polynomials of degree at most $d \geq 2$. Let $b$ be the Krull dimension and $c = r+1-b$ be the codimension of $k[x_0,\dots,x_r]/I$. Then
  \[ \varphi(\HP_I) \leq \left( \frac{3}{2} d^c + d \right)^{b 2^{b-1}}. \]
\end{thm}

Now, we are ready to give an upper bound:

\NNSbound
\begin{proof}
  If $X$ is a projective space or $\dim X \leq 1$, then $(\NNS X)_\tor$ is trivial. Therefore, we may assume that $d \geq 2$, $2 \leq \dim X \leq r-1$. Let $n = dr$ and $m = (dr)^{r^2 2^{r-1}}$. Then $n \geq (d-1)\codim X$, and \Cref{thm:Gotzmann bound} with $2 \leq b \leq r$ implies that
  \begin{align*}
    \max\{ \varphi(nH), \varphi(X)\}
    &\leq \left(\frac{3}{2}(dr)^{r-1}+dr \right)^{r2^{r-1}} \leq (dr)^{r^2 2^{r-1}} = m.
  \end{align*}
  Therefore, \Cref{thm:bound NS by Hilb} implies that
  \[\order (\NNS X)_\tor \leq 
    \dim_k \Gamma\left(\Hilb_{\HP_{mH}} X\right). \]
  Let $t = (dr)^{r^4 2^{2r-2}}$. \Cref{thm:Gotzmann bound} implies that
  \begin{align*}
    \max\{ \varphi(mH), d\}
    &\leq \left(\frac{3}{2}m^{r-1}+m \right)^{r2^{r-1}} \leq m^{r^2 2^{r-1}} = t.
  \end{align*}
  Therefore, \Cref{lem:Hilb degree} implies that $\Hilb_{\HP_{mH}} X$ is defined by polynomials of degree at most $D = P(t+1)+1$ in $\mathbb{P}(\bigwedge^{P(t)}R_t)$. Let $N = \dim_k \bigwedge^{P(t)}R_t$, and let $P$ be the Hilbert polynomial of $\mathscr{I}_{mH}$. Because $r \geq 3$ and $t \geq 6^{1295}dr$,
  \[ D = P(t+1)+1 \leq \binom{t+r+1}{r} +1 \leq t^r. \]
  and
  \[ N = \binom{\dim R_t}{P(t)} \leq 2^{\dim R_t} = 2^{\binom{t+r}{r}} \leq 2^{t^r/4}. \]
  Therefore, \Cref{thm:bound NS by Hilb} and \Cref{thm:bound on Gamma} implies that
  \[\order(\NNS X)_{\tor} \leq \Gamma\left(\Hilb_{\HP_{mH}} X\right) \leq 2^N(D^2/2+D)^{N2^{N-1}} \leq (2D)^{N2^N}. \]
  Furthermore,
  \begin{align*}
    \log_2 (2D)^{N2^N}
    &\leq N2^N(1 + \log_2 D) \leq N \cdot 2^N \cdot (1 + r\log_2 t) \leq 2^N \cdot 2^N \cdot 2^{t^r/2} \leq 2^{2^{t^r}}\\
    \log_d \log_2 \log_2 2^{2^{t^r}}
    &= r \log_d t = r^52^{2r-2}(1+\log_d r) \leq r^62^{2r-2}.
  \end{align*}
  As a resulut,
  \begin{equation}\label{eqn:better NNS bound}
    \# (\NNS X)_{\tor} \leq \exp_2\exp_2\exp_2\exp_d\exp_2(2r+6\log_2 r-2).
    \qedhere
  \end{equation}
\end{proof}

In the rest of this section, we drop the condition that $X$ is connected. Let $Y_0,Y_1,\dots,Y_{n-1}$ be the connected components of $X$.

\begin{thm}
  Let $X \hookrightarrow \mathbb{P}^r$ be a smooth projective variety defined by homogeneous polynomials of degree $\leq d$. Then
  \[ \#(\NNS X)_{\tor} \leq  \exp_2\exp_2\exp_2\exp_d\exp_2(2r+7\log_2 r). \]
\end{thm}
\begin{proof}
  Since
  \[\Pic X = \Pic Y_0 \times \Pic Y_1 \dots \times \Pic Y_{n-1}, \]
  we have
  \[ (\NNS X)_\tor = (\NNS Y_0)_\tor \times (\NNS Y_1)_\tor \dots \times (\NNS Y_{n-1})_\tor. \]
  The Andreotti--Bézout inequality \cite[Lemma~1.28]{Cat} implies that $n \leq d^r$ and $\deg Y_i \leq d^r$ for every $i$. Moreover, every $Y_i$ is defined by homogeneous polynomials of degree at most $\deg Y_i$ by \cite[Proposition~3]{Hei}. Without loss of generality, we may assume that
  \[\#(\NNS Y_0)_\tor = \max \{ \#(\NNS Y_i)_\tor \,|\, 0 \leq i < n\}.\]
  Then

\begin{align*}
  \# (\NNS X)_\tor
  &\leq \left(\# (\NNS Y_0)_\tor\right)^{d^r}\\
  &\leq \left(\exp_2\exp_2\exp_2\exp_{d^r}\exp_2(2r+6\log_2 r-2)\right)^{d^r} \text{ (by (\ref{eqn:better NNS bound}))}\\
  &\leq \exp_2\exp_2\exp_2\left((d^r)^{r^6 2^{2r-2}}{d^r}\right)\\
  &\leq \exp_2\exp_2\exp_2\exp_d \left(r^7 2^{2r}\right)\\
  &= \exp_2\exp_2\exp_2\exp_d\exp_2(2r+7\log_2 r). \qedhere
\end{align*}  
\end{proof}

\section{Application to Fundamental groups}\label{sec:fundamental group}
Let $X$ be a smooth connected projective variety with base point $x_0 \in X(k)$. If $k = \mathbb{C}$, then the torsion abelian group $(\pic X)_\tor$ is the dual of the profinite abelian group $\pi^\et_1(X,x_0)^\ab$. More explicitly,
\[ \pi^\et_1(X,x_0)^\ab = \Hom((\pic^\tau X)_\tor,\mathbb{Q}/\mathbb{Z}) = \varprojlim_{n > 0} (\pic^\tau X)[n]^\dual. \]
This can be generalized to algebraically closed fields $k$ of arbitrary characteristic by using $\Pic^\tau X$ and Nori's fundamental group scheme $\piNr(X,x_0)$. Before proceeding, note that $\pi^\et_1(X,x_0)^\ab = \piNr(X,x_0)^\ab(k)$ due to \cite[Lemma~3.1]{EHS}.
\begin{thm}[Antei {\cite[Proposition 3.4]{Ant}}]\label{thm:pi pic duality}
  The commutative torsion group scheme $(\Pic^\tau X)_\tor$ is the Cartier dual of the commutative profinite group scheme $\piNr(X,x_0)^\ab$. More precisely,
  \[  \piNr(X,x_0)^\ab = \varprojlim_{n > 0} (\Pic^\tau X)[n]^\dual.\]
\end{thm}

If $k = \mathbb{C}$, then $\pi^\et_1(X,x_0)^\ab_\tor \simeq (\NS X)_\tor^\dual$. We will show that $\piNr(X,x_0)^\ab_\tor \simeq (\NNS X)_\tor^\dual$ for general base fields. Then \Cref{thm:NNS bound} gives an upper bound on $\#\piNr(X,x_0)^\ab_\tor$.

\begin{lem}\label{lem:group lem 0}
  Let $A$ be a proper commutative group scheme. Let $A^0$ be the identity component of $A$. Then for every large and divisible $m$, we have $m A_\tor = A^0_{\red,\tor}$.
\end{lem}
\begin{proof}
  Notice that $A^0_\red$ is an abelian variety and $A/A^0_\red$ is a finite group scheme. Suppose that $m$ divides $\#(A/A^0_\red)$. Since multiplication by m is surjective on the abelian variety $A^0_\red$, we get $mA = A^0_\red$. As a result, $m A = A^0_\red$, meaning that $m A_\tor = A^0_{\red,\tor}$.
\end{proof}

\begin{lem}\label{lem:group lem 1}
  Let $A$ be a commutative group scheme of finite type. Let $B$ be a subgroup scheme of $A$. Suppose that $B$ is an abelian variety. Then for every positive integer $n$, the natural morphism
  \[ \varphi\colon \frac{A[n]}{B[n]} \rightarrow \frac{A}{B}[n] \]
  is an isomorphism. 
\end{lem}
\begin{proof}
  By the snake lemma,
  \[\begin{tikzcd}
      0 \arrow[r] & B\arrow[r]\arrow[d,"\times n"] & A\arrow[r]\arrow[d,"\times n"]
      & A/B\arrow[r]\arrow[d,"\times n"] & 0\\
      0 \arrow[r] & B\arrow[r]                     & A\arrow[r]
      & A/B\arrow[r]                     & 0
    \end{tikzcd}\]
  gives the exact sequence
  \[ 0 \longrightarrow B[n] \longrightarrow A[n] \longrightarrow \frac{A}{B}[n] \longrightarrow \frac{B}{nB}. \]
  Since $B$ is an abelian variety, we have $B/nB$ = 0. 
\end{proof}


\begin{thm}\label{cor:pi et NS comparison}
  Let $X$ be a smooth connected projective variety with base point $x_0 \in X(k)$. Then
  \begin{align*}
    \piNr(X,x_0)^\ab_\tor &\simeq (\NNS X)_\tor^\dual.
  \end{align*}
\end{thm}
\begin{proof}
  Keep in mind that limits commute with kernels and colimits commute with cokernels. Keep in mind that the Cartier duality is a contravariant equivalence. \Cref{thm:pi pic duality} implies that
  \begingroup
  \allowdisplaybreaks
  \begin{align*}
    \piNr(X,x_0)^\ab_\tor
    &= \varinjlim_{m>0} \left(\varprojlim_{n > 0} (\Pic^\tau X)[n]^\dual\right)[m] \\
    &= \varinjlim_{m>0} \varprojlim_{n > 0} \left((\Pic^\tau X)[n]^\dual[m]\right) \\
    &= \varinjlim_{m>0} \varprojlim_{n > 0} \left(\frac{(\Pic^\tau X)[n]}{m(\Pic^\tau X)[n]}\right)^\dual \\
    &= \varinjlim_{m>0} \left(\varinjlim_{n > 0}\frac{(\Pic^\tau X)[n]}{m(\Pic^\tau X)[n]}\right)^\dual \\
    &= \varinjlim_{m>0} \left(\frac{\varinjlim_{n > 0}(\Pic^\tau X)[n]}{\varinjlim_{n > 0}m(\Pic^\tau X)[n]}\right)^\dual \\
    &= \varinjlim_{m>0}\left(\frac{(\Pic^\tau X)_\tor}{m(\Pic^\tau X)_\tor}\right)^\dual \\
    &= \left(\varprojlim_{m>0}\frac{(\Pic^\tau X)_\tor}{m(\Pic^\tau X)_\tor}\right)^\dual \\
    &= \left(\frac{(\Pic^\tau X)_\tor}{(\Pic^0 X)_{\red,\tor}}\right)^\dual \text{ (by \Cref{lem:group lem 0})} \\
    &= \left(\frac{\varinjlim_{n > 0}(\Pic^\tau X)[n]}{\varinjlim_{n > 0}(\Pic^0 X)_\red[n]}\right)^\dual \\
    &= \left(\varinjlim_{n > 0}\frac{(\Pic^\tau X)[n]}{(\Pic^0 X)_\red[n]}\right)^\dual \\
    &= \left(\varinjlim_{n > 0}(\NNS X)_\tor [n]\right)^\dual \text{ (by \Cref{lem:group lem 1})}\\
    &= (\NNS X)_\tor^\dual. \qedhere
  \end{align*}
  \endgroup
\end{proof}

Therefore, we obtain an upper bound on $\# \piNr(X,x_0)^\ab_\tor$.

\begin{thm}\label{thm:nr bound}
  Let $X \hookrightarrow \mathbb{P}^r$ be a smooth connected projective variety defined by homogeneous polynomials of degree $\leq d$ with base point $x_0 \in X(k)$. Then
  \[ \# \piNr(X,x_0)^\ab_\tor \leq \exp_2\exp_2\exp_2\exp_d\exp_2(2r+6 \log_2 r). \]
\end{thm}
\begin{proof}
  \Cref{cor:pi et NS comparison} implies that $\# \piNr(X,x_0)^\ab_\tor = \# (\NNS X)_\tor$. Therefore, the bound follows from \Cref{thm:NNS bound}.
\end{proof}

 Suppose that $\charac k = p > 0$ and $\ell \neq p$ is a prime number. Since $\pi^\et_1(X,x_0)^\ab = \piNr(X,x_0)^\ab(k)$, \Cref{cor:pi et NS comparison} implies that $\pi^\et_1(X,x_0)^\ab[\ell^\infty] = (\NS X)[\ell^\infty]^\dual$. However, this is not true for $p$-power torsions. Hence, a bound on $\# (\NS X)_\tor$ such as \cite[Theorem~4.12]{Kwe} does not give a bound on $\#\pi^\et_1(X,x_0)^\ab_\tor$.

\piBound
\begin{proof}
  By \cite[Proposition 69]{Jak}, we have $\pi^\et_1(X,x_0)^\ab_\tor = (\NNS X)_\tor^\dual(k)$. Therefore,
  \[ \# \pi^\et_1(X,x_0)^\ab_\tor = \# (\NNS X)_\tor^\dual(k) \leq \# (\NNS X)_\tor, \]
  and the bound follows from \Cref{thm:NNS bound}.
\end{proof}

\section{Lefschetz Hyperplane Theorem}\label{sec:Lefschetz}
Throughout the section, $X \hookrightarrow \mathbb{P}^r$ is a smooth connected projective variety satisfying $\dim X \geq 2$, and $H$ is a smooth hyperplane section of $X$. In this section, we discuss the Lefschetz hyperplane theorem for $\pic^\tau X$.

If $k = \mathbb{C}$, the exponential sequence gives a short exact sequence
\[ 0 \longrightarrow \frac{H^1(X,\mathbb{Z})\otimes_{\mathbb{Z}}\mathbb{R}}{H^1(X,\mathbb{Z})}
  \longrightarrow \pic^\tau X \longrightarrow H^2(X,\mathbb{Z})_\tor \longrightarrow 0. \]
With the natural analytic topology, the Pontryagin duality gives an exact sequence
\[ 0 \longrightarrow H_1(X,\mathbb{Z})_\tor
  \longrightarrow \Hom_\cont(\pic^\tau X,\mathbb{R}/\mathbb{Z})
  \longrightarrow \frac{H_1(X,\mathbb{Z})}{H_1(X,\mathbb{Z})_\tor} \longrightarrow 0. \]
  The Lefschetz hyperplane theorem implies that the map $H_1(H,\mathbb{Z}) \rightarrow H_1(X,\mathbb{Z})$ is surjective. Therefore, the lemma below implies that $\pic^\tau X \rightarrow \pic^\tau H$ is injective.

  \begin{lem}
    Let $f\colon A \rightarrow B$ be a surjective morphism between finitely generated abelian groups. Consider the morphism between exact sequences
    \[\begin{tikzcd}
        0 \arrow[r] & A_\tor \arrow[r]\arrow[d,"f_*"] & A' \arrow[r]\arrow[d,"g"]
        & A/A_\tor \arrow[r]\arrow[d,"f_*"] & 0\\
        0 \arrow[r] & B_\tor \arrow[r]                      & B' \arrow[r]
        & B/B_\tor \arrow[r]              & 0
      \end{tikzcd}\]
    where the outer vertical morphisms are induced by $f$. Then $g\colon A' \rightarrow B'$ is surjective.
  \end{lem}
  \begin{proof}
    This follows from the snake lemma.
  \end{proof}

  We expect a similar Lefschetz-type theorem for any algebraically closed field $k$.
  Over a general base field, we have the Lefschetz hyperplane theorem for \'etale fundamental groups \cite[XII. Corollaire 3.5]{SGA2}. Hence, for a prime number $\ell \neq \charac k$, $(\pic X)[\ell^\infty] \rightarrow (\pic H)[\ell^\infty]$ is injective by \Cref{thm:pi pic duality}. However, \'etale fundamental groups lack the information of $(\pic X)[p^\infty]$. 
  
\begin{thm}\label{thm:semi Lefschetz}
  Let $X$ be a smooth connected projective variety, and let $H$ be a smooth hyperplane section of $X$. If $\dim X \geq 2$, then the kernel of
  \[ r\colon\Pic^\tau X \rightarrow \Pic^\tau H\]
  is a finite commutative group scheme with a connected Cartier dual. 
\end{thm}
\begin{proof}
  Let $\ell \neq \charac k$ be a prime number. The connected component $(\ker r)^0_\red$ of $(\ker r)_\red$ is an abelian variety without $\ell$-torsion points. Therefore, $(\ker r)^0_\red = 0$ and $\ker r$ is a finite commutative group scheme.

  Take a base point $x \in H$, and let $M = \#(\ker r)$. Then
  \begin{align*}
    (\ker r)^\dual &= \ker(\Pic^\tau X \rightarrow \Pic^\tau H)^\dual \\
            &= \ker\big((\Pic X)[M] \rightarrow (\Pic H)[M]\big)^\dual \\
            &= \coker\left((\Pic H)[M]^\dual \rightarrow \Pic X)[M]^\dual\right)\\
            &= \coker\left(\frac{\piNr(H,x)^\ab}{M \piNr(H,x)^\ab} \rightarrow
              \frac{\piNr(X,x)^\ab}{M \piNr(X,x)^\ab}\right) \text{ (by \Cref{thm:pi pic duality})}.
  \end{align*}
  We have
  \[\frac{\piNr(X,x)^\ab}{M \piNr(X,x)^\ab}(k) = \frac{\pi^\et_1(X,x)^\ab}{M \pi^\et_1(X,x)^\ab}\]
  by \cite[Lemma~3.1]{EHS}. Moreover,  
  \[ \frac{\pi^\et_1(H,x)^\ab}{M \pi^\et_1(H,x)^\ab} \rightarrow
    \frac{{\pi}^\et_1(X,x)^\ab}{M {\pi}^\et_1(X,x)^\ab} \]
  is surjective by the Lefschetz hyperplane theorem for \'etale fundamental groups \cite[XII. Corollaire~3.5]{SGA2}.
  Thus, $(\ker r)^\dual(k) = 0$, meaning that $(\ker r)^\dual$ is connected.
\end{proof}

One may want to show that $r\colon \Pic^\tau X \rightarrow \Pic^\tau H$ is injective. Unfortunately, the Lefschetz hyperplane theorem for Nori's fundamental group scheme is no longer true \cite[Remark 2.4]{BH}. Nonetheless, we still have the Lefschetz hyperplane theorem in some cases.
  
\begin{thm}\label{thm:very very ample}
  Let $d$ be a sufficiently large integer. Then for any smooth hyperplane section $H'$ of the $d$-uple embedding of $X \hookrightarrow \mathbb{P}^r$, the natural map
  \[ r\colon\Pic^\tau X \rightarrow \Pic^\tau H'\]
  is a closed embedding.
\end{thm}
\begin{proof}
  Let $M = \#(\ker r)$. If d is sufficiently large, then the natural map $\piNr(H',x) \rightarrow \piNr(X,x)$ is faithfully flat by \cite[Theorem~1.1]{BiHo}. In this case,  
\[(\ker r)^\dual = \coker\left(\frac{\piNr(H,x)^\ab}{M \piNr(H,x)^\ab} \rightarrow
              \frac{\piNr(X,x)^\ab}{M \piNr(X,x)^\ab}\right)
            = 0\]
          as in the proof of \Cref{thm:semi Lefschetz}.
\end{proof}

\begin{thm}[Langer {\cite[Theorem~11.3]{Lan}}]\label{thm:W2 Lef}
  Suppose that $X$ has a lifting to a smooth projective scheme over $W_2(k)$. Then 
  \[ r\colon\Pic^\tau X \rightarrow \Pic^\tau H\]
  is a closed embedding.
\end{thm}

The author does not know whether or not $\pic^\tau X \rightarrow \pic^\tau H$ is injective in general. However, the author conjectures that $\Pic^\tau X \rightarrow \Pic^\tau H$ can fail to be a closed embedding if $\charac k = p > 0$, because of the failure of the Kodaira vanishing theorem. The argument below is a reformulation of the work in \cite[Section 2]{BiHo} and \cite[Example~10.1]{Lan}.

\begin{defi}
  Let $\boldsymbol{\alpha}_{p^n}$ be the kernel of the $p^n$-power Frobenius endomorphism on $\mathbb{G}_a$.  
\end{defi}
\begin{thm}
  Let $D$ be a smooth ample effective divisor on $X$. Then the natural map
  \[ r(\boldsymbol{\alpha}_p)\colon(\Pic^\tau X)(\boldsymbol{\alpha}_{p}) \rightarrow (\Pic^\tau D)(\boldsymbol{\alpha}_{p})\]
  is injective if and only if $H^1(X,\mathcal{O}_X(-D)) = 0$.
\end{thm}
\begin{proof}
  Suppose that $H^1(X,\mathcal{O}_X(-D)) \neq 0$. Because $\dim X \geq 2$, there is a large integer $n$ such that $H^1(X,\mathcal{O}_X(-p^n D)) = 0$. Let $F_X\colon X \rightarrow X$ be the absolute Frobenius morphism. Then $(F_X^n)^*\mathcal{O}_X(-D) = \mathcal{O}_X(-p^n D)$. Thus, we have the diagram with exact rows and columns as below.
\[\begin{tikzcd}
    && 0\arrow[d] & 0\arrow[d]\\
    && H^1_{\mathrm{fppf}}(X,\boldsymbol{\alpha}_{p^n})\arrow[r,"h"]\arrow[d,"i"]
    & H^1_{\mathrm{fppf}}(D,\boldsymbol{\alpha}_{p^n})\arrow[d]\\
    0 \arrow[r]
    & H^1(X,\mathcal{O}_X(-D))\arrow[r,"f"]\arrow[d,"(F_X^n)^*"]
    & H^1(X,\mathcal{O}_X)\arrow[r,"g"]\arrow[d,"(F_X^n)^*"]
    & H^1(D,\mathcal{O}_D)\\
    0 \arrow[r] & H^1(X,\mathcal{O}_X(-p^n D))\arrow[r] & H^1(X,\mathcal{O}_X)
  \end{tikzcd}\]
Take a nonzero $s \in H^1(X,\mathcal{O}_X(-D))$. Then $(F_X^n)^*(f(s)) = 0$, because $H^1(X,\mathcal{O}_X(-p^n D)) = 0$. Thus, there is a nonzero $t \in H^1_{\mathrm{fppf}}(X,\boldsymbol{\alpha}_{p^n})$ such that $f(s) = i(t)$. On the other hand, $h(t) = 0$, since $g(f(s)) = 0$. Therefore, $h$ is not injective. By \cite[Proposition~3.2]{Ant}, the natural morphism
\[r(\boldsymbol{\alpha}_{p^n}) \colon (\Pic^\tau X)(\boldsymbol{\alpha}_{p^n}) \rightarrow (\Pic^\tau D)(\boldsymbol{\alpha}_{p^n}) \]
is isomorphic to $h$, so is also not injective. Let $\varphi\colon\boldsymbol{\alpha}_{p^n} \rightarrow \Pic^\tau X$ be a nonzero element of $\ker r(\boldsymbol{\alpha}_{p^n})$. Then the image of $\varphi$ is $\boldsymbol{\alpha}_{p^m}$ for some $m > 0$. Furthermore, $\boldsymbol{\alpha}_{p^m}$ has a subgroup isomorphic to $\boldsymbol{\alpha}_p$. Thus, we have the commutative diagram below.
\[\begin{tikzcd}
    &\boldsymbol{\alpha}_{p^n}\arrow[rd,"\varphi"]\arrow[d,two heads]&&\\
    \boldsymbol{\alpha}_{p}\arrow[r,hook]&\boldsymbol{\alpha}_{p^m}\arrow[r,hook]& \Pic^\tau X\arrow[r]& \Pic^\tau D
  \end{tikzcd}\]
The second row gives a nonzero element in the kernel of
\[ r(\boldsymbol{\alpha}_p)\colon(\Pic^\tau X)(\boldsymbol{\alpha}_{p}) \rightarrow (\Pic^\tau D)(\boldsymbol{\alpha}_{p}).\]

Now, suppose that $H^1(X,\mathcal{O}_X(-D)) = 0$. Then
$H^1(X,\mathcal{O}_X)\rightarrow H^1(D,\mathcal{O}_D)$ is injective, and so is $H^1_{\mathrm{fppf}}(X,\boldsymbol{\alpha}_{p})\rightarrow H^1_{\mathrm{fppf}}(D,\boldsymbol{\alpha}_{p})$. By \cite[Proposition~3.2]{Ant}, 
\[ r(\boldsymbol{\alpha}_p)\colon(\Pic^\tau X)(\boldsymbol{\alpha}_{p})
  \rightarrow (\Pic^\tau D)(\boldsymbol{\alpha}_{p})\]
is also injective.
\end{proof}

Raynaud \cite{Ray} gave an ample effective divisor $D$ such that $H^1(X,\mathcal{O}_X(-D)) \neq 0$. Lauritzen \cite{Laur2}\cite{Laur} showed that the Kodaira vanishing theorem can fail even if $D$ is very ample. However, the author does not know any example of a very ample divisor $D$ such that $H^1(X,\mathcal{O}_X(-D)) \neq 0$.

Thus we still do not know the answers to the following.

\begin{qst}
  Let $X$ be a smooth connected projective variety of $\dim X \geq 2$, and let $H$ be a smooth hyperplane section. Is the map $\pic^\tau X \rightarrow \pic^\tau H$ always injective? Is the map $\Pic^\tau X \rightarrow \Pic^\tau H$ always a closed embedding?
\end{qst}

\section{The Number of Generators of \texorpdfstring{$(\NS X)_{\tor}$}{(NS X)\_tor}}\label{sec:generator bound}
Let $\dim X \hookrightarrow \mathbb{P}^r$ be a smooth connected projective variety such that $\dim X \geq 1$. The aim of this section is to prove that $(\NS X)_{\tor}$ is generated by at most $(\deg X-1)(\deg X-2)$ elements, if $X$ has a lifting over $W_2(k)$.

In the rest of the section, let $C$ be the intersection of $X$ with $\dim X - 1$ general hyperplanes in $\mathbb{P}^r$. Then $C$ is a smooth connected curve, and the genus $g(C)$ of $C$ is at most $(\deg X-1)(\deg X-2)/2$. Since $\Pic^\tau C = \Jac C$, we have a natural map
\[ r\colon \Pic^\tau X \rightarrow \Jac C. \]

\begin{thm}\label{cor:Pic tau embedding}
  Let $C$ be a general curve section. Then $\ker r$ is a finite commutative group scheme with a connected dual.
\end{thm}
\begin{proof}
  By \Cref{thm:semi Lefschetz} applied repeatedly, $r$ is a composition of homomorphisms with finite connected kernel, so $\ker r$ also is finite and connected.
\end{proof}

\begin{lem}\label{lem:w2 lef}
  If $X\hookrightarrow\mathbb{P}^r$ is the reduction of a smooth connected projective scheme $\mathcal{X}\hookrightarrow\mathbb{P}_{W_2(k)}^r$ over $W_2(k)$, then $\ker r = 0$.
\end{lem}
\begin{proof}
  Let $V\hookrightarrow\mathbb{P}^r$ be a general hyperplane, and let $H = X \cap V$. Let $\mathcal{V}\hookrightarrow\mathbb{P}_{W_2(k)}^r$ be a lifting of $V$. Then $\mathcal{V}\cap\mathcal{X}$ is a smooth projective scheme, and $H$ is its reduction. Therefore, if $\dim X \geq 1$, then the hypothesis on $X$ is inherited by a general hyperplane section $H$. We may now apply \Cref{thm:W2 Lef} iteratively to obtain the result.
\end{proof}

\begin{thm}\label{thm:Lef Gen Bound}
  If $\ker r$ is connected, then $(\NS X)_{\tor}$ is generated by $(\deg X-1)(\deg X-2)$ elements.
\end{thm}
\begin{proof}
  Let $N = \# (\NS X)_{\tor}$. Then $(\NS X)_{\tor}$ is a quotient group of $(\pic X)[N]$. Since $\ker r$ is connected,
  \[ \pic^\tau X \rightarrow (\Jac C)(k)\]
  is injective. Thus,
  \[ (\pic^\tau X)[N] \rightarrow (\Jac C)(k)[N]\]
  is also injective.

  The group $(\Jac C)(k)[N]$ is generated by $2 g(C)$ elements, and $2 g(C) \leq (\deg X-1)(\deg X-2)$. Thus, its subquotient $(\NS X)_{\tor}$ is also generated by $\leq (\deg X-1)(\deg X-2)$ elements.
\end{proof}

\NSGenLifting
\begin{proof}
  This follows from \Cref{lem:w2 lef} and \Cref{thm:Lef Gen Bound}.
\end{proof}

Furthermore, the $p$-power torsion subgroups have smaller upper bounds.

\begin{thm}
  Suppose that $\charac k = p > 0$. Then $(\NNS X)[p^\infty]^\dual(k)$ is generated by at most $(\deg X - 1)(\deg X - 2)/2$ elements. If $\ker r$ is connected, then $(\NS X)[p^\infty]$ is also generated by at most $(\deg X - 1)(\deg X - 2)/2$ elements.
\end{thm}
\begin{proof}
  Take a sufficiently large integer $N$. Let $(\Jac C)^\dual$ be the dual abelian variety of $(\Jac C)$. Since $(\Jac C)^\dual[p^N] = (\Jac C)[p^N]^\dual$, 
  \[ (\Jac C)^\dual[p^N](k) \rightarrow (\Pic X)[p^N]^\dual(k)\]
  is surjective by \Cref{cor:Pic tau embedding}. Because $N$ is large, $(\NNS X)[p^\infty]^\dual$ is a subgroup scheme of $(\Pic X)[p^N]^\dual$. Since $(\Jac C)^\dual(k)[p^N]$ is generated by $\leq(\deg X - 1)(\deg X - 2)/2$ elements, so is $(\NNS X)[p^\infty]^\dual(k)$.

  Now, suppose that $\ker r$ is connected. Then we have an injection
  \[ (\pic^\tau X)[p^N] \rightarrow (\Jac C)(k)[p^N]. \]
  Because $N$ is large, $(\NS X)[p^\infty]$ is a quotient of $(\pic^\tau X)[p^N]$. Since $(\Jac C)(k)[p^N]$ is generated by $\leq(\deg X - 1)(\deg X - 2)/2$ elements, so is $(\NS X)[p^\infty]$.
\end{proof}

\section*{Acknowledgement}
The author thanks his advisor Bjorn Poonen for his careful advice. The author thanks J\'anos Koll\'ar for answering questions regarding Lefschetz-type theorems. The author thanks Barry Mazur for suggesting Nori's fundamental group schemes. The author also thanks Chenyang Xu, Steven Kleiman and Davesh Maulik for many helpful conversations.

\bibliographystyle{plain}
\bibliography{mybib}

\begin{thebibliography}{10}

\bibitem{Ant}
Marco Antei.
\newblock On the abelian fundamental group scheme of a family of varieties.
\newblock {\em Israel J. Math.}, 186:427--446, 2011.

\bibitem{BH}
Indranil Biswas and Yogish~I. Holla.
\newblock Comparison of fundamental group schemes of a projective variety and
  an ample hypersurface.
\newblock {\em J. Algebraic Geom.}, 16(3):547--597, 2007.

\bibitem{BiHo}
Indranil Biswas and Yogish~I. Holla.
\newblock Comparison of fundamental group schemes of a projective variety and
  an ample hypersurface.
\newblock {\em J. Algebraic Geom.}, 16(3):547--597, 2007.

\bibitem{Cat}
Fabrizio Catanese.
\newblock Chow varieties, {H}ilbert schemes and moduli spaces of surfaces of
  general type.
\newblock {\em J. Algebraic Geom.}, 1(4):561--595, 1992.

\bibitem{SGA3I}
M.~Demazure, A.~Grothendieck, and M.~Artin.
\newblock {\em Sch{\'e}mas en groupes (SGA 3): Propri{\'e}t{\'e}s
  g{\'e}n{\'e}rales des sch{\'e}mas en groupes}.
\newblock Documents math{\'e}matiques. Soci{\'e}t{\'e} mathematique de France,
  2011.

\bibitem{Dub}
Thomas~W. Dub\'{e}.
\newblock The structure of polynomial ideals and {G}r\"{o}bner bases.
\newblock {\em SIAM J. Comput.}, 19(4):750--775, 1990.

\bibitem{Eis2}
David Eisenbud.
\newblock {\em Commutative Algebra: with a view toward algebraic geometry},
  volume 150.
\newblock Springer Science \& Business Media, 2013.

\bibitem{EHS}
H\'{e}l\`ene Esnault, Ph\`ung~H\^{o} Hai, and Xiaotao Sun.
\newblock On {N}ori's fundamental group scheme.
\newblock In {\em Geometry and dynamics of groups and spaces}, volume 265 of
  {\em Progr. Math.}, pages 377--398. Birkh\"{a}user, Basel, 2008.

\bibitem{Got}
Gerd Gotzmann.
\newblock Eine {B}edingung f\"{u}r die {F}lachheit und das {H}ilbertpolynom
  eines graduierten {R}inges.
\newblock {\em Math. Z.}, 158(1):61--70, 1978.

\bibitem{EGA4}
Alexander Grothendieck.
\newblock {\'E}l{\'e}ments de g{\'e}om{\'e}trie alg{\'e}brique: {IV}. {\'e}tude
  locale des sch{\'e}mas et des morphismes de sch{\'e}mas, seconde partie.
\newblock {\em Publications Math{\'e}matiques de l'IH{\'E}S}, 24:5--231, 1965.

\bibitem{SGA2}
Alexander Grothendieck.
\newblock {\em Cohomologie locale des faisceaux coh\'{e}rents et
  th\'{e}or\`emes de {L}efschetz locaux et globaux {$(SGA$} {$2)$}}.
\newblock North-Holland Publishing Co., Amsterdam; Masson \& Cie, \'{E}diteur,
  Paris, 1968.
\newblock Augment\'{e} d'un expos\'{e} par Mich\`ele Raynaud, S\'{e}minaire de
  G\'{e}om\'{e}trie Alg\'{e}brique du Bois-Marie, 1962, Advanced Studies in
  Pure Mathematics, Vol. 2.

\bibitem{Harris}
Joe Harris.
\newblock {\em Algebraic geometry}, volume 133 of {\em Graduate Texts in
  Mathematics}.
\newblock Springer-Verlag, New York, 1992.
\newblock A first course.

\bibitem{Har}
Robin Hartshorne.
\newblock {\em Algebraic geometry}.
\newblock Springer-Verlag, New York-Heidelberg, 1977.
\newblock Graduate Texts in Mathematics, No. 52.

\bibitem{Hei}
Joos Heintz.
\newblock Definability and fast quantifier elimination in algebraically closed
  fields.
\newblock {\em Theoret. Comput. Sci.}, 24(3):239--277, 1983.

\bibitem{HeHi}
J\"{u}rgen Herzog and Takayuki Hibi.
\newblock {\em Monomial ideals}, volume 260 of {\em Graduate Texts in
  Mathematics}.
\newblock Springer-Verlag London, Ltd., London, 2011.

\bibitem{Hoa}
L\^{e}~Tu\^{a}n Hoa.
\newblock Finiteness of {H}ilbert functions and bounds for
  {C}astelnuovo-{M}umford regularity of initial ideals.
\newblock {\em Trans. Amer. Math. Soc.}, 360(9):4519--4540, 2008.

\bibitem{IaKa}
Anthony Iarrobino and Vassil Kanev.
\newblock {\em Power sums, {G}orenstein algebras, and determinantal loci},
  volume 1721 of {\em Lecture Notes in Mathematics}.
\newblock Springer-Verlag, Berlin, 1999.
\newblock Appendix C by Iarrobino and Steven L. Kleiman.

\bibitem{Igu}
Jun-ichi Igusa.
\newblock On some problems in abstract algebraic geometry.
\newblock {\em Proc. Nat. Acad. Sci. U.S.A.}, 41:964--967, 1955.

\bibitem{Kle2}
Steven~L. Kleiman.
\newblock The {P}icard scheme.
\newblock In {\em Fundamental algebraic geometry}, volume 123 of {\em Math.
  Surveys Monogr.}, pages 235--321. Amer. Math. Soc., Providence, RI, 2005.

\bibitem{Kol}
J\'{a}nos Koll\'{a}r.
\newblock {\em Rational curves on algebraic varieties}, volume~32 of {\em
  Ergebnisse der Mathematik und ihrer Grenzgebiete. 3. Folge. A Series of
  Modern Surveys in Mathematics [Results in Mathematics and Related Areas. 3rd
  Series. A Series of Modern Surveys in Mathematics]}.
\newblock Springer-Verlag, Berlin, 1996.

\bibitem{Kwe}
Hyuk~Jun Kweon.
\newblock Bounds on the torsion subgroups of néron-severi groups, 2019.
\newblock \url{https://arxiv.org/pdf/1902.02753.pdf}.

\bibitem{Lan}
Adrian Langer.
\newblock On the {S}-fundamental group scheme.
\newblock {\em Ann. Inst. Fourier (Grenoble)}, 61(5):2077--2119 (2012), 2011.

\bibitem{Laur2}
N.~Lauritzen and A.~P. Rao.
\newblock Elementary counterexamples to {K}odaira vanishing in prime
  characteristic.
\newblock {\em Proc. Indian Acad. Sci. Math. Sci.}, 107(1):21--25, 1997.

\bibitem{Laur}
Niels Lauritzen.
\newblock Embeddings of homogeneous spaces in prime characteristics.
\newblock {\em Amer. J. Math.}, 118(2):377--387, 1996.

\bibitem{MaMe}
Ernst~W. Mayr and Albert~R. Meyer.
\newblock The complexity of the word problems for commutative semigroups and
  polynomial ideals.
\newblock {\em Adv. in Math.}, 46(3):305--329, 1982.

\bibitem{Mum}
David Mumford.
\newblock {\em Lectures on curves on an algebraic surface}.
\newblock With a section by G. M. Bergman. Annals of Mathematics Studies, No.
  59. Princeton University Press, Princeton, N.J., 1966.

\bibitem{Ner}
Andr\'{e} N\'{e}ron.
\newblock Probl\`emes arithm\'{e}tiques et g\'{e}om\'{e}triques rattach\'{e}s
  \`a la notion de rang d'une courbe alg\'{e}brique dans un corps.
\newblock {\em Bull. Soc. Math. France}, 80:101--166, 1952.

\bibitem{Nor}
Madhav~V. Nori.
\newblock The fundamental group-scheme.
\newblock {\em Proc. Indian Acad. Sci. Math. Sci.}, 91(2):73--122, 1982.

\bibitem{PTL}
Bjorn Poonen, Damiano Testa, and Ronald van Luijk.
\newblock Computing {N}\'{e}ron-{S}everi groups and cycle class groups.
\newblock {\em Compos. Math.}, 151(4):713--734, 2015.

\bibitem{Ray}
M.~Raynaud.
\newblock Contre-exemple au ``vanishing theorem'' en caract\'{e}ristique
  {$p>0$}.
\newblock In {\em C. {P}. {R}amanujam---a tribute}, volume~8 of {\em Tata Inst.
  Fund. Res. Studies in Math.}, pages 273--278. Springer, Berlin-New York,
  1978.

\bibitem{Sev}
Francesco Severi.
\newblock La base per le variet\`a algebriche di dimensione qualunque contenuta
  in una data, e la teoria generate delle corrispond\'enze fra i punti di due
  superficie algebriche.
\newblock {\em Mem. R. Accad. Italia}, 5:239--83, 1933.

\bibitem{Jak}
Jakob Stix.
\newblock {\em Rational points and arithmetic of fundamental groups}, volume
  2054 of {\em Lecture Notes in Mathematics}.
\newblock Springer, Heidelberg, 2013.
\newblock Evidence for the section conjecture.

\end{thebibliography}

\end{document}